\documentclass[leqno, 10pt]{amsart}
\usepackage{amsmath, amsthm, amssymb, latexsym}
 \newtheorem{definition}{Definition}[section]
 
 \newtheorem{lemma}[definition]{Lemma}
 \newtheorem{proposition}[definition]{Proposition}

 \newtheorem*{theorem*}{Theorem}
\newtheorem*{proposition*}{Proposition}
\newtheorem*{lemma*}{Lemma}

 \theoremstyle{remark}

  \newtheorem*{acknowledgements}{Acknowledgements}

\newcommand{\op}[1]{\operatorname{#1}}

\newcommand{\Tr}{\ensuremath{\op{Tr}}}

\newcommand{\C}{\ensuremath{\mathbb{C}}} 
\newcommand{\bH}{\ensuremath{\mathbb{H}}} 
\newcommand{\N}{\ensuremath{\mathbb{N}}} 
\newcommand{\R}{\ensuremath{\mathbb{R}}} 
\newcommand{\Z}{\ensuremath{\mathbb{Z}}}

\newcommand{\fa}{\ensuremath{\mathfrak{a}}}
\newcommand{\fg}{\ensuremath{\mathfrak{g}}}
\newcommand{\fh}{\ensuremath{\mathfrak{h}}}
\newcommand{\fk}{\ensuremath{\mathfrak{k}}}
\newcommand{\fl}{\ensuremath{\mathfrak{l}}}
\newcommand{\fm}{\ensuremath{\mathfrak{m}}}
\newcommand{\fp}{\ensuremath{\mathfrak{p}}}

\newcommand{\cF}{\ensuremath{\mathcal{F}}}

\newcommand{\cJ}{\ensuremath{\mathcal{J}}}

\newcommand{\cL}{\ensuremath{\mathcal{L}}}

\newcommand{\cN}{\ensuremath{\mathcal{N}}}
\newcommand{\cS}{\ensuremath{\mathcal{S}}}

\newcommand{\End}{\ensuremath{\op{End}}}

\newcommand{\codim}{\op{codim}}

\newcommand{\Ad}{\op{Ad}}

\renewcommand{\Box}{\square}

\newcommand{\dbar}{\overline{\partial}}
\newcommand{\dbarb}{\overline{\partial}_{H}}

\newcommand{\Boxb}{\Box_{b}}

\begin{document}
 \begin{abstract}
     The aim of this paper is to present the construction, out of the Kohn-Rossi complex, of a new hypoelliptic operator $Q_{L}$ on almost CR manifolds equipped with a real structure. 
     The operator acts on all $(p,q)$-forms, but when restricted to $(p,0)$-forms and $(p,n)$-forms it is a sum of squares up to sign factor and lower order terms. 
     Therefore, only a finite type condition condition is needed to have hypoellipticity on those forms. However, outside these forms $Q_{L}$
    may fail to be hypoelliptic, as it is shown in the example of the Heisenberg group $\bH^{5}$.
\end{abstract}

\title{A New Hypoelliptic Operator\\ on Almost CR Manifolds} 

\author{Rapha\"el Ponge}

\address{Department of Mathematics, University of Toronto, 40 St George Street, Toronto, ON, M5S 2E4, Canada.}
\address[Current Address]{Graduate School of Mathematical Sciences, University of Tokyo, 3-8-1 Komaba, Meguro-ku,
Tokyo 153-8914, Japan.}
\email{ponge@ms.u-tokyo.ac.jp}
\thanks{Research partially supported by NSERC Discovery Grant 341328-07 (Canada) and 
by a New Staff Matching Grant of the Connaught Fund of the University of Toronto (Canada)}
\maketitle 
\numberwithin{equation}{section}

\maketitle 

 \section*{Introduction}
Homogeneity reasons prevent natural operators on CR manifolds to be elliptic, but they can be hypoelliptic in various other guises. An important 
example is the Kohn Laplacian: under suitable geometric conditions (i.e., $Y(q)$-condition) the Kohn Laplacian is maximal 
hypoelliptic and hypoelliptic with a gain of 1 derivative (see~\cite{Ko:BCM}, \cite{FS:EDdbarbCAHG}, \cite{BG:CHM}), but in general it 
may have rather subtle hypoelliptic properties (see \cite{Ch:RPdEWPCMD3}, \cite{Ch:OdE3DCRM}, \cite{FK:HEDCD23CRM}, 
\cite{FKM:HECRMDLF}, \cite{Ko:MSHETCROBL}, \cite{Ko:EdCPCRM}, \cite{NS:dCDBCn}). 

 The aim of this paper is to present the construction of a new hypoelliptic operator on almost CR manifolds, that is, manifolds $M$ together with a 
subbundle $H\subset TM$ which is equipped with an almost complex structure $J\in \End H$, $J^{2}=-1$. This construction is partly inspired by the 
second order signature operator of Connes-Moscovici~\cite{CM:LIFNCG} and an earlier version on 3-dimensional contact manifolds was presented in~\cite{Po:CRAS2}.

In order to construct our operator it is crucial to further assume that the horizontal subbundle $H$ admits a real structure, i.e., there exists a subbundle $L\subset H$ so that $H=L\oplus JL$. 
This  implies the vanishing of the first Chern class of $H$, and so an almost CR manifold does not admit a real structure in 
general. However, as explained in Section~\ref{sec.CRARS}, there is a handful of interesting examples of CR manifolds which do admit a real 
structure. Among these are real hypersurfaces of $\C^{n+1}$ that are rigid in the sense of~\cite{BRT:CRSGAECRF}, nilpotent Lie groups and CR nilmanifolds, some 
CR symmetric spaces in the sense of~\cite{KZ:SCRM}, and contact manifolds equipped with a Legendrian subbundle, including circle bundles associated to the geometric 
quantization of symplectic manifolds.

The existence of a real structure allows us to define a chirality operator which is analogous to the Hodge $\star$-operator and 
maps $(p,q)$-forms to $(n-p,n-q)$-forms (where $n$ is the complex dimension of $\dim H$). We then can define a second order differential operator by 
letting 
\begin{equation*}
    Q_{L}=(\bar\partial_{H}^{*}\bar\partial_{H} - 
     \bar\partial_{H}\bar\partial_{H}^{*}) - \gamma (\bar\partial_{H}^{*}\bar\partial_{H} - 
     \bar\partial_{H}\bar\partial_{H}^{*})\gamma,
\end{equation*}
where $\dbarb$ is the horizontal $\dbar$-operator of Kohn-Rossi~(\cite{KR:EHFBCM}, \cite{Ko:BCM}). This operators acts on $(p,q)$-forms and anticommutes with the chirality operator 
$\gamma$. 

When restricted to $(p,0)$-forms 
and $(p,n)$-forms $Q_{L}$ agrees with a sum of squares up to sign factor and lower order terms (see Proposition~\ref{prop:QL.local-form}). 
Therefore, whenever $M$ has finite type, on these forms $Q_{L}$ is maximal hypoelliptic, which in this context implies that the operator is hypoelliptic with gain of 
one derivative (see Section~\ref{sec:hypoelliticity}). In fact, it even admits a parametrix in the class of singular-integral operators of 
Rotschild-Stein~\cite{RS:HDONG} and, when $\codim H=1$, it further has a parametrix in the Heisenberg calculus of Beals-Greiner~\cite{BG:CHM} and 
Taylor~\cite{Ta:NCMA}. Notice that in order for all these properties to hold only the finite type condition is needed. In particular, when $\codim H=1$ we may allow $(M,H)$ to be weakly 
pseudoconvex. 

The hypoelliptic properties of $Q_{L}$ contrast with that of the Kohn Laplacian. For instance, on strictly pseudoconvex CR manifolds 
$Q_{L}$ is hypoelliptic precisely in bidegrees where the Kohn Laplacian is not. In particular, in dimension 3 the operator $Q_{L}$ is 
hypoelliptic in every bidegree, while the Kohn Laplacian is hypoelliptic in none. In addition, on weakly pseudoconvex CR manifolds 
that are not strictly pseudoconvex $Q_{L}$ may be maximal hypoelliptic, while the Kohn Laplacian may not. 

On the other hand, outside $(p,0)$-forms and $(n,0)$-forms the operator $Q_{L}$ may fail to be hypoelliptic. This fact is illustrated in 
Section~\ref{sec:non-hypoellipticity}, where we look at the operator $Q_{L}$ on the 5-dimensional Heisenberg group $\bH^{5}$. In this setting we explicitly construct a 
$(0,1)$-form which annihilates $Q_{L}$ but is singular at the origin. This shows that $Q_{L}$ is not hypoelliptic on 
$(0,1)$-forms. Similar arguments also show that $Q_{L}$ is not hypoelliptic on $(1,1)$-forms or on $(2,1)$-forms either. Therefore, in the case of $\bH^{5}$ 
this is only on $(p,0)$-forms and $(p,n)$-forms that $Q_{L}$ is hypoelliptic.

This paper is organized as follows. 
In Section~\ref{sec.CRARS}, we present the main definitions and examples regarding real structures on almost CR 
manifolds. 
In Section~\ref{sec:chirality}, we construct the 
chirality operator $\gamma$ mentioned above. 
In Section~\ref{sec:QL}, we construct the operator $Q_{L}$ and we derive a local expression which shows that on $(p,0)$-forms and $(p,n)$-forms $Q_{L}$ 
is a sum of squares up to sign factor and lower order terms. 
In Section~\ref{sec:hypoelliticity}, we study the hypoellipticity properties of $Q_{L}$ on $(p,0)$-forms 
and $(p,n)$-forms. 
In Section~\ref{sec:non-hypoellipticity}, we look at the operator $Q_{L}$ on the Heisenberg group $\bH^{5}$ and illustrate on this example the fact that $Q_{L}$ 
may fail to be hypoelliptic outside forms of bidegree $(p,0)$ and $(p,n)$. 

\begin{acknowledgements} 
   I am very grateful to Olivier Biquard, Louis Boutet de Monvel, Alain Connes, Charlie Epstein, Henri Moscovici, Xiang Tang and Alan Weinstein for helpful and stimulating 
    discussions related to the subject matter of this paper. 
\end{acknowledgements}

\section{Real structures on CR manifolds}\label{sec.CRARS}
Let $M$ be an almost CR manifold, i.e., $M$ is equipped with a subbundle $H\subset TM$ carrying an almost complex  structure $J 
\in C^{\infty}(M,H)$, $J^{2}=-1$. This gives rise to a CR structure when $J$ is integrable, i.e., the subbundle $T_{1,0}:=\ker 
(J+i)\subset T_{\C}M$ is integrable in Froebenius' sense. 

We don't assume $H$ to have codimension 1, that is, $(M,H)$ need not be of 
hypersurface type. In any case $n:=\frac{1}{2}\dim H$ is an integer, called the \emph{CR dimension} of $M$. 

In addition, we shall say that $M$ is of \emph{finite 
type}, when H\"ormander's bracket condition is satisfied, i.e., at every point $TM$ is spanned by successive Lie brackets $[X_{1},[X_{2},[\ldots,X_{m}]\ldots]]$ 
of vectors fields with values in $H$.

\begin{definition}
     A real structure on $H$ is given by the datum of a rank $n$ real subbundle $L\subset H$ such that 
      \begin{equation}
H=L\oplus \cJ L.
          \label{eq:CR.def-real-sructure}
      \end{equation}
 \end{definition}
 
If $L$ is real structure on $H$, then the decomposition~(\ref{eq:CR.def-real-sructure}) yields an involution $X\rightarrow \underline{X}$ on the fibers of $H$ defined by
\begin{equation}
    \underline{X_{1}+\cJ Y_{1}}=X_{1}-\cJ Y_{1} \qquad \forall X_{1},Y_{1}\in C^{\infty}(M,L).
         \label{eq:CR.involution}
\end{equation} 
Notice that  $L=\ker (\underline{.}-1)$ and $\cJ L=\ker 
(\underline{.}+1)$. Conversely, if $\iota$  is an involutive section of $\End_{\R}H$ anticommuting with $\cJ$, then the subbundle 
$L:=\ker (\iota-1)$ defines a real structure on $H$.

Let us now look at some examples of real structures. 

\subsubsection{Rigid real hypersurfaces} 
Let us denote $(z^{1},\ldots,z^{n},w)$ the complex coordinates on $\C^{n+1}$. In terms of real and imaginary parts we 
shall write $z^{j}=x^{j}+iy^{j}$, $j=1,\ldots,n$, and $w=u+iv$. Consider a real hypersurface of the form,
\begin{equation}
    M=\{ v=F(z,\bar{z})\},
\end{equation}
where $F(z,\bar{z})$ is some real-valued function. In the terminology of~\cite{BRT:CRSGAECRF} such a hypersurface is said to be \emph{rigid}. Examples of 
such hypersurfaces are given by the hyperquadrics 
$Q_{p,q}^{2n+1}:=\{ v= \sum_{j=1}^{p}|z^{j}|^{2}-\sum_{j=p+1}^{q}|z^{j}|^{2}\}$, $p+q=n$. 

We equip $M$ with the CR structure induced by the complex structure of $\C^{n+1}$, i.e., the differential $J_{0}$ of the multiplication by $i$ on 
$T\C^{n+1}$. Therefore we have $H=TM\cap J_{0}(TM)$ and the complex structure of $H$ is just $J=J_{0|M}$. Then the CR tangent bundle $T_{1,0}:=\ker(J-i)$ agrees 
with $T^{1,0}\C^{n+1}\cap T_{\C}M$. In particular, a global frame of $T_{1,0}$ is given by the vector fields, 
\begin{equation}
    Z_{j}=\frac{\partial}{\partial z^{j}}+i \partial_{z^{j}} F(z,\bar{z})\frac{\partial}{\partial w} \qquad j=1,\ldots,n.
\end{equation}

For $j=1,\ldots,n.$ set $Z_{j}=X_{j}-iY_{j}$, where $X_{j}$ and $Y_{j}$ denote the real and imaginary parts of $Z_{j}$, i.e.,
\begin{equation}
    X_{j}=\frac{\partial}{\partial x^{j}}+ \partial_{y^{j}} F \frac{\partial}{\partial u}+ \partial_{x^{j}} F \frac{\partial}{\partial v}, \qquad 
    Y_{j}=\frac{\partial}{\partial y^{j}} -\partial_{x^{j}}F\frac{\partial}{\partial u}+ \partial_{y^{j}}F\frac{\partial}{\partial v}.
\end{equation}
Then the vector fields $X_{1},\ldots,X_{n}$ and $Y_{1},\ldots,Y_{n}$ form a frame of $H$ such that $JX_{j}=Y_{j}$. Therefore, if we let $L$ denote the subbundle 
spanned by $X_{1},\ldots, X_{n}$ then $H=L\oplus JL$, that is, $L$ defines a real structure on $H$.

\subsubsection{Nilpotent Lie groups and CR nilmanifolds} Let $\bH^{2n+1}$ denote the $(2n+1)$-dimensional  
Heisenberg group. We realize $\bH^{2n+1}$ as $\R \times \R^{2n}$ equipped with the 
group law, 
\begin{equation}
    x.y=(x^{0}+y^{0}+\sum_{1\leq j\leq n}(x^{n+j}y^{j}-x^{j}y^{n+j}),x^{1}+y^{1},\ldots,x^{2n}+y^{2n}).
    \label{eq:real-structures.Heisenberg-law}
\end{equation}
This group law is homogeneous with respect to the parabolic dilations, 
\begin{equation}
    t.(x^{0},x^{1},\ldots,x^{2n})=(t^{2}x^{0},tx^{1},\ldots,tx^{2n}) \qquad t>0. 
     \label{eq:RealStructure.Heisenberg-dilations}
\end{equation}

Identifying the Lie algebra $\fh^{2n+1}$ of $\bH^{2n+1}$ with the Lie algebra of left-invariant vector fields, a basis for $\fh^{2n+1}$ is 
provided by the left-invariant vector-fields,
\begin{equation}
    X_{0}=\frac{\partial}{\partial x^{0}}, \quad X_{j}=\frac{\partial}{\partial x^{j}}+x^{n+j}\frac{\partial}{\partial 
    x^{0}}, \quad X_{n+j}=\frac{\partial}{\partial x^{n+j}}-x^{j}\frac{\partial}{\partial 
    x^{0}}, 
     \label{eq:Examples.Heisenberg-left-invariant-basis}
\end{equation}
where $j$ ranges from $1$ to $n$.  In particular, for $j,k=1,\ldots,n$ and $k\neq j$ we have the relations
$[X_{j},X_{n+k}]=-2\delta_{jk}X_{0}$ and $[X_{0},X_{j}]=[X_{j},X_{k}]=[X_{n+j},X_{n+k}]=0$. 
Notice also that with respect to the 
dilations~(\ref{eq:RealStructure.Heisenberg-dilations}) 
the vector fields $X_{0}$ is homogeneous of degree $-2$, while $X_{1},\ldots,X_{2n}$ are homogeneous of degree $-1$.

Let $H$ be the subbundle spanned by $X_{1},\ldots,X_{2n}$. We endow $H$ with the complex structure $\cJ$ such that 
$\cJ X_{j}=X_{n+j}$ and $\cJ X_{n+j}=-X_{j}$. 
This defines a homogeneous left-invariant CR structure on $\bH^{2n+1}$. A left-invariant real 
structure on $H$ is then provided by the subbundle $L$ spanned by $X_{1},\ldots,X_{n}$.


More generally, let $G$ be a real nilpotent Lie group which is 
stratified, in the sense that its Lie algebra $\fg:=T_{0}G$ admits a grading by vector subspaces,
\begin{equation}
    \fg=\fh_{1}\oplus \ldots \oplus \fh_{k}, 
\end{equation}
such that $\fh_{j}:=[\fh_{1},\fh_{j-1}]$, $j=2,\ldots,k$. Set $\fh=\fh_{1}$ and assume that $\dim _{\R}\fh$ is even, say $\dim_{\R}\fh=2n$. 
As a real vector space $\fh$ is isomorphic 
to $\C^{n}$, and so it 
admits a complex structure $\cJ_{0}$. Let $H$ be the subbundle of left-invariant vector fields such that $H_{|x=0}=\fh$ and let $\cJ$ be the almost complex 
structure on $H$ such that $\cJ_{|x=0}=\cJ_{0}$. Then $H$ and $\cJ$ define a  left-invariant almost CR structure on $G$. 

Let $\fl$ be an $n$-dimensional real subspace of $\fh$ such that $\fh=\fl\oplus \cJ_{0}\fl$, and let $L$ denote the subbundle of left-invariant vector 
fields such that $L_{|x=0}=\mathfrak{l}$. Then $H=L\oplus \cJ L$, and so $L$ gives rise to a left-invariant real structure on $H$. 

Let $\Gamma \subset G$ be a lattice, i.e., a discrete cocompact subgroup of $G$. Then \mbox{$M:=\Gamma \backslash G$} is a 
compact nilmanifold. Since $H$ and $\cJ$ are left-invariant, and $\Gamma$ acts on left, they descend to $M$, and hence
define a natural CR structure on $M$. Similarly $L$ descends to a vector bundle on $M$, and thereby gives rise to a real structure on $M$.

\subsubsection{CR Symmetric Spaces} Let $G$ be a connected Lie group with Lie algebra $\fg$ and let $j \in G$ have order $4$. Let $\tau$ be the 
automorphism $\tau(x)=jxj^{-1}$, $x \in G$, and let $G_{\tau}$ denote the fixed point group of $\tau$. Set $s=j^{2}$ and $\sigma=\tau^{2}$, so that 
$\sigma$ is an involutive automorphism of $G$. 

Let $K$ be a compact subgroup of $G_{\tau}$ with Lie algebra $\fk$ and let $\fl$ be an $\op{Ad}(K)$-invariant subspace of 
$ \ker (\op{Ad}(s)+ 1)$ such that $\fl \cap \Ad(j)\fl=\{0\}$. Set $\fh=\fl\oplus \Ad(j)\fl$; this is an  $\op{Ad}(K)$-invariant subspace. 
Let $\fa$ denote the Lie algebra generated by $\fh$ and let us further assume that $\fg=\fk+\fa$. In addition, since $K$ is compact and $\fh$ is $\Ad(K)$-invariant, there is a subspace 
$\fp$ of $\fg$ containing $\fh$ such that $\fg=\fk\oplus \fp$. 

 Under these conditions $M:=G/K$ is a CR symmetric space in the sense of Kaup-Zaitsev (see~\cite[Sect.~6]{KZ:SCRM}).  
  Let $o\in M$ denote the class of $1$. Then $T_{o}M$ is isomorphic to 
$\fp$ and under this identification there are a unique $G$-invariant bundle $H\subset TM$ and a unique $G$-invariant almost complex $\cJ$ structure on 
$H$ such that $H_{o}=\fh$ and $\cJ_{o}=\Ad(j)|_{\fh}$. In addition, the condition $\fg=\fk+\fa$, where $\fa$ is the Lie algebra of $\fh$,  insures us 
that this CR structure is of finite type (\emph{cf}.~\cite[Prop.~6.2]{KZ:SCRM}). 

Let $L$ denote the unique $G$-invariant subbundle of $H$ such that $L_{0}=\fl$, then we 
have $H=L\oplus JL$. Therefore $L$ defines a $G$-invariant real structure on $M$. 

The above construction can be illustrated by the following example. We take $G=SU(n)$,  so that the Lie algebra  $\fg=\mathfrak{su}(n)$ consists of trace-free 
skew-Hermitian matrices. In addition, we let $p$ and $q$ be positive integers so that $n=p+q$. Identifying 
$\C^{n}$ with $\C^{p}\oplus\C^{q}$ we shall write $n\times n$-matrices as block matrices 
$\left( \begin{array}{cc}
    a & b  \\
    c & d
\end{array}\right)$. 
Set $j:=\left( 
\begin{array}{cc}
    1 & 0  \\
     0 & -i
\end{array}\right)$. This is a 4th-order element of $SU(n)$ and we have $s:=j^{2}=\left( 
\begin{array}{cc}
    1 & 0  \\
     0 & -1
\end{array}\right)$. Thus, 
\begin{gather}
    \tau \left[ \left( \begin{array}{cc}
    a & b  \\
    c & d
\end{array}\right) \right]=\Ad(j) \left( \begin{array}{cc}
    a & b  \\
    c & d
\end{array}\right) = \left( \begin{array}{cc}
    a & ib  \\
    -ic & d
\end{array}\right),\\
      \sigma \left[ \left( \begin{array}{cc}
    a & b  \\
    c & d
\end{array}\right) \right]=\Ad(s) \left( \begin{array}{cc}
    a & b  \\
    c & d
\end{array}\right) = \left( \begin{array}{cc}
    a & -b  \\
    -c & d
\end{array}\right) . 
\end{gather}
In particular, we have the splitting $\mathfrak{su}(n)=\fm_{+}\oplus \fm_{-}$, where 
\begin{gather}
   \fm_{+} :=\ker(\Ad(s)-1)=\left\{
         \left( \begin{array}{cc}
    a & 0  \\
    0 & d
\end{array}\right); a\in \mathfrak{u}(p), d\in \mathfrak{u}(q), \Tr a+\Tr d=0 \right\},  \\
   \fm_{-} :=\ker(\Ad(s)+1)=\left\{
         \left( \begin{array}{cc}
    0 & b  \\
    -b^{*} & 0
\end{array}\right) \right\}.
\end{gather}

Set $ K:=SO(p)\times SO(q)=\left\{
         \left( \begin{array}{cc}
    a & 0  \\
    0 & d
\end{array}\right); a\in SO(p), d\in SO(q) \right\}$. This is a (compact) subgroup of $G_{\tau}$. Define $  \fl:=\fm_{-}\cap M_{n}(\R)$, i.e., 
\begin{equation}
    \fl=\left\{
         \left( \begin{array}{cc}
    0 & b  \\
    -b^{t} & 0
\end{array}\right); b\in M_{p,q}(\R) \right\}.
\end{equation}
Then $\fl$ is an $\Ad(K)$-invariant subspace such that  $\fl \cap \Ad(j)\fl=\{0\}$. In fact, we have $\fh:=\fl \oplus \Ad(j)\fl=\fm_{-}$ and, 
using elementary matrices, it is not difficult to check that 
$[\fm_{-},\fm_{-}]=\fm_{+}$. Therefore $\fg$ agrees with the Lie algebra generated by $\fh$. It then follows that $SU(n)/(SO(p)\times SO(q))$ is a CR 
symmetric space of finite type with a $SU(n)$-invariant real structure defined by $\fl$. 

\subsubsection{Contact manifolds, Legendrian subbundles and geometric quantization} Assume that $(M^{2n+1},H)$ is an orientable contact manifold, 
i.e., $H$ is the annihilator of a 
globally defined contact form $\theta$ on $M$. 
Let $\cJ$ be an almost complex 
 structure on $H$ which is calibrated with respect to $d\theta_{|H}$, i.e., $g_{\theta,\cJ}(X,Y):=d\theta(X,\cJ Y)$ is a positive-definite metric on $H$ 
 (since $M$ is orientable such an almost complex structure always exists).  In particular $(H,\cJ)$ defines an almost CR structure on $M$.
 
 Let $L$ be a Legendrian subbundle of $H$, i.e., $L$ is a maximal isotropic subbundle of 
 $d\theta_{|H}$.  If $X$ and $Y$ are sections  of $L$, then we have
\begin{equation}
    g_{\theta,\cJ}(X,\cJ Y)=d\theta(X,\cJ^{2}Y)=-d\theta(X,Y)=\theta([X,Y])=0.
    \label{eq:Examples.Levi-metric-Legendrian}
\end{equation}
 This implies the orthogonal decomposition $H=L\oplus \cJ L$, and so $L$ defines a real structure independently of the choice of $\cJ$. 
 
 Conversely, suppose that $L$ is a subbundle $H$ of rank $n$ such that $L$ and $\cJ L$ are orthogonal. As 
 in~(\ref{eq:Examples.Levi-metric-Legendrian})  if $X$ and $Y$ are sections  of $L$, then
 \begin{equation}
     d\theta(X,Y)=-\theta([X,Y])= -g_{\theta,\cJ}(X,\cJ Y)=0.
 \end{equation}
 Thus $L$ is a Legendrian subbundle of $H$. 
 
 As we shall now recall examples of contact manifolds equipped with a Legendrian subbundle naturally occur in the context of the geometric quantization of symplectic 
 manifolds. 
 
 Let $(X^{2n},\omega)$ be a 
 symplectic manifold which is \emph{prequantizable}, i.e., the cohomology class of $\frac{1}{2\pi}\omega$ is integral (see, e.g., \cite{AE:QMGPA}, \cite{Wo:GQ}). 
 Then there exists a Hermitian line bundle $(L,h)$ on $X$ with a metric connection $\nabla^{L}$ with curvature $F^{L}=-i\omega$. 
 Let $L^{*}$ denote the dual line bundle with unit sphere bundle 
$S^{*}(L):=\{\xi\in L^{*}; h(\xi,\xi)=1\}$. Recall that the connection 1-form of the dual connection 
 $\nabla^{L^{*}}$  makes sense as a globally defined real 1-form on the total space $M$ of $S(L^{*})$. This can be seen as follows. 
 
 Let $\xi$ be a local section of $S(L^{*})$. Regarding $\xi$ as a non-zero section of $L^{*}$ it defines a local trivialization of $L^{*}$ with 
 respect to which we have $\nabla^{L^{*}}=d+i\alpha$, where  $\alpha:=-ih(\xi,\nabla^{L*}\xi)$. 
 Notice that, as $\nabla^{L^{*}}$ is a metric connection, $\alpha$ must be a real 1-form. Let $\lambda$ denote the 
 local fiber coordinate on $M$ defined by $\xi$ and consider the \emph{real} 1-form defined by
 \begin{equation}
     \theta:=  p^{*}\alpha -i\lambda^{-1}d\lambda,
      \label{eq:Examples.connection-1-form}
 \end{equation}
where $p:M\rightarrow X$ is the fibration of $M$ over $X$. 

Let $\xi'$ be another local section of $S(L^{*})$ and set $\alpha':=-ih(\xi',\nabla^{L*}\xi')$ and $\xi'=\mu \xi$, where $\mu$ is an $S^{1}$-valued function. 
 Then $\alpha'=\alpha-i\mu^{-1}d\mu$. The local fiber coordinate on $M$ defined by $\xi'$ is $\lambda'=(p^{*}\mu)^{-1}\lambda$, 
so the 1-form~(\ref{eq:Examples.connection-1-form}) corresponding to $\xi'$ is 
\begin{multline}
    \theta'=p^{*}\alpha'-i\lambda^{'-1}d\lambda^{'}=p^{*}[\alpha-i \mu^{-1}d\mu]-i [(p^{*}\mu)^{-1}\lambda]^{-1}d[(p^{*}\mu)^{-1}\lambda]\\ = 
    p^{*}\alpha -i\lambda^{-1}d\lambda=\theta.
\end{multline}
This shows that the 1-form $\theta$ in~(\ref{eq:Examples.connection-1-form}) does not depend on the choice of the local section $\xi$, and so it makes sense globally on $M$. 
 
 Let $H=\ker \theta$ and let $V:=\ker dp \subset TM$ be the vertical bundle of $M$. It follows from~(\ref{eq:Examples.connection-1-form}) that  $H\cap V=\{0\}$, and so $dp$ induces an 
 isomorphism from $H$ onto $TX$. Using~(\ref{eq:Examples.connection-1-form}) we also see that locally $d\theta=p^{*}d\alpha$. By assumption we have 
 $id\alpha=F^{L^{*}}=-F^{L}=i\omega$, and so 
 $d\theta=p^{*}\omega$. Since $dp$ induces an 
 isomorphism from $H$ onto $TX$,  it follows that $d\theta$ is non-degenerate on $H$, i.e., $\theta$ is a contact form on $H$. 
 
 Let us further assume that $(X,\omega)$ is \emph{quantizable} in the sense that it admits a Lagrangian subbundle, i.e, a subbundle $\Lambda\subset 
 TX$ which is maximal isotropic for $\omega$ (see, e.g., \cite[Sect.~3.2]{AE:QMGPA}, \cite[Sect.~4.5]{Wo:GQ}). 
 Then the subbundle $L:=p^{*}\Lambda \cap H$ is maximal isotropic for $d\theta$, i.e., $L$ is 
 a Legendrian subbundle of $M$. 
  Therefore, we see that the quantization of a symplectic manifold naturally gives rise to a contact manifold $(M,\theta)$ equipped with a 
  Legendrian subbundle.
 
\section{Real structure and chirality operator}\label{sec:chirality}
Throughout the rest of the paper we let $(M,H)$ be an almost CR manifold of CR dimension $n$. We also assume $H$ to have a real 
structure $L$, that is, 
\begin{equation}
    H=L\oplus \cJ L,
\end{equation}
where $\cJ$ denotes the almost complex structure of $H$. 

In addition, we endow $L$ with a Riemannian metric $g_{L}$. Extending $g_{L}$ to be zero on $\cJ L\times H$ and $H\times \cJ L$, we endow $H$ with the 
Riemannian metric, 
\begin{equation}
    g_{H}:=g_{L}(.,.)+g_{L}(\cJ.,\cJ.).
\end{equation}
With respect to this metric $\cJ$ becomes an isometry and the splitting~(\ref{eq:CR.def-real-sructure}) becomes orthogonal. If $(M,H)$ is of finite type, then we see that 
$(H,g_{H})$ defines a sub-Riemannian structure on $M$ compatible with its almost CR structure. 

We fix a choice of supplement $\cN$ of $H$ in $TM$. This allows us to identify $H^{*}$ with the annihilator of $\cN$ in $T^{*}M$. We also set 
$T_{1,0}=\ker(\cJ +i)$ and $T_{0,1}=\ker(\cJ -i)$. Notice that $\overline{T_{1,0}}=T_{0,1}$. Moreover, extending $g_{H}$ into a Hermitian metric on $T_{\C}M$ we 
get the orthogonal decomposition,
\begin{equation}
    H\otimes \C=T_{1,0}\oplus T_{0,1}.
\end{equation}

For $p,q=0,\ldots,n$ we let $\Lambda^{p,q}:=(\Lambda^{1,0})^{p}\wedge (\Lambda^{0,1})^{q}$ denote the bundle of $(p,q)$-covectors, where 
$\Lambda^{1,0}$  and $\Lambda^{0,1}$ are the respective annihilators 
in $T^{*}_{\C}M$ of the subbundles $T_{0,1}\oplus  (\cN\otimes \C)$ and $T_{1,0}\oplus  (\cN\otimes \C)$. We then have the orthogonal splitting, 
\begin{equation}
    \Lambda^{*}_{\C}H^{*}=\bigoplus_{p,q=0}^{n}\Lambda^{p,q}. 
     \label{eq:CR-Lambda-pq-decomposition}
\end{equation}

We shall now turn the bundle of $(p,q)$-covectors into a super-bundle by equipping it with a suitable chirality 
operator. To define this operator we 
shall make use of the real structure of $H$. To this end we
extend the involution~(\ref{eq:CR.involution}) into the antilinear involution on $H\otimes \C$ defined by
 \begin{equation}
     \underline{X+iY}=\underline{X}-i \underline{Y} \qquad \forall X,Y\in C^{\infty}(M,H). 
 \end{equation}
This involution preserves both  $T_{1,0}$ and $T_{0,1}$. Therefore, by duality it gives rise to an antilinear involution of $\Lambda^{*,*}$ 
preserving the bidegree. As we shall see the latter property will be crucial in the construction of the operator $Q_{L}$ in the next section. 

Let $v_{H}(x)$ be the volume form of $g_{H}$ (seen as a section of $\Lambda^{n,n}$), and let $g^{*}_{H}$ denote the Hermitian metric on 
$\Lambda^{*,*}$ induced by $g_{H}$. The operator $\star:\Lambda^{*,*}\rightarrow \Lambda^{n-*,n-*}$ is uniquely determined by the formula,
 \begin{equation}
    \beta\wedge  \underline{\star\alpha}= g^{*}_{H}(\beta,\alpha) v_{H}(x)  \qquad \forall \alpha,\beta \in C^{\infty}(M,\Lambda^{p,q}).
      \label{eq:CR.Hodge-operator}
 \end{equation}

 Let $X_{1},\ldots,X_{n}$ be an orthonormal frame of $L$. Since the splitting $H=L\oplus \cJ L$ is orthogonal we see that $\{X_{j},\cJ X_{j}\}$ is an 
 orthonormal frame of $H$. For $j=1,\ldots,n$ we set $Z_{j}=\frac{1}{\sqrt{2}}(X_{j}-i\cJ X_{j})$ and $Z_{\bar{j}}=\frac{1}{\sqrt{2}}(X_{j}+i\cJ X_{j})$. Then 
 $\{Z_{j}\}$ and $\{Z_{j}, Z_{\bar{j}}\}$ are orthonormal frames of  $T_{1,0}$ and $H\otimes \C$ respectively. Any orthonormal frame of $H\otimes \C$ obtained by 
 a similar process is said to be $\emph{admissible}$. Notice that the invariance of $L$ under the involution~(\ref{eq:CR.involution})
 imply that 
 $\underline{Z_{j}}=Z_{j}$ and $\underline{Z_{\bar{j}}}=Z_{\bar{j}}$.
 
 Let $\{\theta^{j},\theta^{\bar{j}}\}$ be the coframe of $H^{*}\otimes \C$ dual to $\{Z_{j}, Z_{\bar{j}}\}$. For any ordered subsets 
 $J=\{j_{1},\ldots,j_{p}\}$ and $K=\{k_{1},\ldots, k_{q}\}$ of $\{1,\ldots,n\}$ with $j_{1}<\ldots<j_{p}$ and $k_{1}<\ldots<k_{q}$ we set 
$\theta^{J,\bar{K}}:=\theta^{1}\wedge \cdots \wedge \theta^{j_{p}}\wedge \theta^{\bar{k_{1}}}\wedge \cdots \wedge \theta^{\bar{k_{q}}}$.
 Then $\{\theta^{J,\bar{K}}\}$ is an orthonormal coframe of $\Lambda^{*,*}$. 
 
 If  $J=\{j_{1},\ldots,j_{p}\}$ is an ordered subset of $\{1,\ldots,n\}$ with ordered complement $J^{c}=\{j_{1}',\ldots,j_{n-p}'\}$, then we let 
$\varepsilon(J,J^{c})$ denote the signature of the permutation 
\begin{equation}
    (j_{1},\ldots,j_{p},j_{1}',\ldots,j_{n-p}')\longrightarrow (1,\ldots,n).
\end{equation}
 Notice that we always have $\varepsilon(J,J^{c})\varepsilon(J^{c},J)=(-1)^{p(n-p)}$, for
 $\varepsilon(J,J^{c})\varepsilon(J^{c},J)$ is the signature of the permutation $(j_{1},\ldots,j_{p},j_{1}',\ldots,j_{n-p}')\rightarrow 
(j_{1}',\ldots,j_{n-p}',j_{1},\ldots,j_{p})$.

 
\begin{lemma}\label{lem:CR.Hodge-operator} 
1) We have
\begin{equation}
    \star\theta^{J,\bar{K}}=i^{n}(-1)^{\frac{n(n-1)}{2}+q(n-p)}\varepsilon(J,J^{c})\varepsilon(K,K^{c}) 
    \theta^{J^{c},\bar{K^{c}}} .
     \label{eq:CR.*thetaJK}
\end{equation}

2) On $\Lambda^{p,q}$  we have 
\begin{equation}
    \star^{2}=(-1)^{n+p+q}.
\end{equation} 
 \end{lemma}
 \begin{proof}
 Set  $\theta^{n,\bar{n}}:=\theta^{1}\wedge \cdots \wedge \theta^{n}\wedge \theta^{\bar{1}}\wedge \cdots \wedge \theta^{\bar{n}}$. Since $H$ is 
 oriented by means of its almost complex structure, locally we have  
 \begin{equation}
     v_{H}(x)=i^{n} \theta^{1}\wedge \theta^{\bar{1}}\wedge \cdots \wedge \theta^{n}\wedge 
      \theta^{\bar{n}}=i^{n}(-1)^{\frac{n(n-1)}{2}}\theta^{n,\bar{n}}.
 \end{equation}
Therefore Eq.~(\ref{eq:CR.Hodge-operator}) can be rewritten as
 \begin{equation}
     \beta\wedge  \underline{\star\alpha}=i^{n}(-1)^{\frac{n(n-1)}{2}}g_{H}^{*}(\beta,\alpha)\theta^{n,\bar{n}}  \qquad \forall 
     \alpha,\beta \in C^{\infty}(M,\Lambda^{p,q}). 
      \label{eq:CR.Hodge-operator2}
 \end{equation}
 
 Let $J_{0}$ and $K_{0}$ be ordered subsets of $\{1,\ldots,n\}$ of respective lengths $p$ and $q$, and set $\omega:=\star\theta^{J_{0},\bar{K_{0}}}=\sum \lambda_{J,\bar{K}} 
 \theta^{J,\bar{K}}$. The $(n,n)$-component of $\underline{\omega}\wedge \Theta^{J,\bar{K}}$ is equal to 
 $\pm\overline{ \lambda_{J,\bar{K^{c}}}}\theta^{n,\bar{n}}$, and so 
 from~(\ref{eq:CR.Hodge-operator2}) we see that $\lambda_{J^{c},\bar{K^{c}}}=0$ unless $J=J_{0}^{c}$ and $K=K_{0}^{c}$, i.e., we have 
 $\omega=\lambda_{J_{0}^{c},\bar{K_{0}^{c}}}\theta^{J_{0}^{c},\bar{K_{0}^{c}}}$. In particular, from~(\ref{eq:CR.Hodge-operator2}) we get
\begin{equation}
 i^{n}(-1)^{\frac{n(n-1)}{2}}\theta^{n,\bar{n}}= \theta^{J_{0},\bar{K_{0}}} \wedge\underline{\omega } = 
\overline{ \lambda_{J_{0}^{c},\bar{K_{0}^{c}}}}\theta^{J_{0},\bar{K_{0}}}\wedge  \theta^{J_{0}^{c},\bar{K_{0}^{c}}}. 
     \label{eq:CR.*thetaJK0}
\end{equation} 

Next, upon writing 
$\theta^{J_{0},\bar{K_{0}}}=\theta^{J_{0},\bar{0}}\wedge \theta^{0,\bar{K_{0}}}$ and 
$\theta^{J_{0}^{c},\bar{K_{0}^{c}}}=\theta^{J_{0}^{c},\bar{0}}\wedge \theta^{0,\bar{K_{0}^{c}}}$ we get
\begin{multline}
 \theta^{J_{0},\bar{K_{0}}} \wedge   \theta^{J_{0}^{c},\bar{K_{0}^{c}}}=
    (-1)^{q(n-p)}\theta^{J_{0},\bar{0}}\wedge\theta^{J_{0}^{c},\bar{0}}\wedge \theta^{0,\bar{K_{0}}} \wedge\theta^{0,\bar{K_{0}^{c}}} \\
    =(-1)^{q(n-p)} \varepsilon(J_{0},J_{0}^{c})\varepsilon(K_{0},K_{0}^{c})\theta^{n,\bar{n}}.
\end{multline}
Comparing this with~(\ref{eq:CR.*thetaJK0}) yields $\lambda_{J_{0}^{c},\bar{K_{0}^{c}}}= 
i^{n}(-1)^{\frac{n(n-1)}{2}+q(n-p)}\varepsilon(J_{0},J_{0}^{c})\varepsilon(K_{0},K_{0}^{c})$. Thus,
\begin{equation}
    \star\theta^{J_{0},\bar{K_{0}}}=i^{n}(-1)^{\frac{n(n-1)}{2}+q(n-p)}\varepsilon(J_{0},J_{0}^{c})\varepsilon(K_{0},K_{0}^{c}) 
    \theta^{J_{0}^{c},\bar{K_{0}^{c}}}.
\end{equation}

Now, using~(\ref{eq:CR.*thetaJK}) we see that $ \star^{2}\theta^{J_{0},\bar{K_{0}}}$ is equal to
\begin{equation}
   (-1)^{n+q(n-p)+(n-q)p}\varepsilon(J_{0}^{c},J_{0})\varepsilon(K_{0}^{c},K_{0}) 
    \varepsilon(J_{0},J_{0}^{c})\varepsilon(K_{0},K_{0}^{c})\theta^{J_{0},\bar{K_{0}}}.
\end{equation}
Since $\varepsilon(J_{0}^{c},J_{0})=(-1)^{(n-p)p} \varepsilon(J_{0},J_{0}^{c})$ and 
$\varepsilon(K_{0}^{c},K_{0})=(-1)^{(n-q)q} \varepsilon(K_{0},K_{0}^{c})$ we get 
$\star^{2}\theta^{J_{0},\bar{K_{0}^{c}}}=(-1)^{n+N}\theta^{J_{0},\bar{K_{0}^{c}}}$, with 
\begin{multline}
    N= q(n-p)+(n-q)p+(n-p)p+(n-q)q=(2n-p-q)(p+q) \\
   =(p+q)^{2}=p+q \quad \bmod 2, 
     \label{eq:CR.N}
\end{multline}
that is, $\star^{2}\theta^{J_{0},\bar{K_{0}}}=(-1)^{n+p+q}\theta^{J_{0},\bar{K_{0}}}$. Thus $\star^{2}=(-1)^{n+p+q}$
on $\Lambda^{p,q}$, proving the lemma.
\end{proof}
 
Next, the chirality operator $\gamma : \Lambda^{*,*}\rightarrow \Lambda^{*,*}$ is defined by
  \begin{equation}
  \gamma=i^{n+(p+q)^{2}}\star \qquad \text{on $\Lambda^{p,q}$}.   
      \label{eq:CR.gamma}
 \end{equation}
Notice that $\gamma$ maps $\Lambda^{p,q}$ onto $\Lambda^{n-p,n-q}$.
 
\begin{lemma}\label{lem:CR.Z2-grading}
The operator $\gamma$ is a $\Z_{2}$-grading, that is, it satisfies
  \begin{equation}
       \gamma^{2}=1 \qquad \text{and} \qquad \gamma^{*}=\gamma. 
  \end{equation}
\end{lemma}
\begin{proof}
First, thanks to Lemma~\ref{lem:CR.Hodge-operator}, on $\Lambda^{p,q}$ we have
\begin{equation}
  \gamma^{2}=    i^{2n+(2n-p-q)^{2}+(p+q)^{2}}\star^{2}=i^{2n+2(p+q)^{2}}(-1)^{n+p+q}=1.
\end{equation}
    
Next, let $\{\theta^{J,\bar{K}}\}$ be the coframe of $\Lambda^{*,*}$ associated to an admissible frame $\{Z_{j},Z_{\bar{j}}\}$ of $H\otimes \C$, 
    and let $J$ and $K$ be ordered subsets of respective lengths $p$ and $q$. 
    Then by~(\ref{eq:CR.*thetaJK}) and~(\ref{eq:CR.gamma}) we have 
    \begin{equation}
        \gamma \theta^{J,\bar{K}}=i^{(p+q)^{2}}(-1)^{\frac{n(n-1)}{2}+q(n-p)}\varepsilon(J,J^{c})\varepsilon(K,K^{c})\theta^{J^{c},\bar{K^{c}}}.
         \label{eq:CR.gamma-thetaJK}
    \end{equation}
    Since $\{\theta^{J,\bar{K}}\}$ is an orthonormal coframe, this gives 
    \begin{equation}
           \gamma^{*} \theta^{J^{c},\bar{K^{c}}}=(-i)^{(p+q)^{2}}(-1)^{\frac{n(n-1)}{2}+q(n-p)}\varepsilon(J,J^{c})\varepsilon(K,K^{c})\theta^{J,\bar{K}}.
    \end{equation}
    Thus for $\theta^{J,\bar{K}}$ we obtain 
    \begin{multline}
         \gamma^{*} 
         \theta^{J,\bar{K}}=(-i)^{(2n-p-q)^{2}}(-1)^{\frac{n(n-1)}{2}+(n-q)p}\varepsilon(J^{c},J)\varepsilon(K^{c},K)\theta^{J^{c},\bar{K^{c}}},\\
         = (-i)^{(p+q)^{2}}(-1)^{\frac{n(n-1)}{2}+(n-q)p+ p(n-p)+q(n-q)}\varepsilon(J,J^{c})\varepsilon(K,K^{c})\theta^{J^{c},\bar{K^{c}}}.
    \end{multline}
    As by~(\ref{eq:CR.N}) we have $(-1)^{(n-q)p+ p(n-p)+q(n-q)}=(-1)^{p+q+q(n-p)}$, we get 
    \begin{equation}
         \gamma^{*} 
         \theta^{J,\bar{K}}=i^{(p+q)^{2}}(-1)^{\frac{n(n-1)}{2}+q(n-p)}\varepsilon(J,J^{c})\varepsilon(K,K^{c})\theta^{J^{c},\bar{K^{c}}} 
         =\gamma \theta^{J,\bar{K}},
    \end{equation}
    that is, $\gamma^{*}=\gamma$. The lemma is thus proved.
\end{proof}
 
\section{The operator $Q_{L}$}\label{sec:QL}
We shall now construct a differential operator $Q_{L}$ acting on the sections of $\Lambda^{*,*}$ 
which is supersymmetric in the sense that it anticommutes with the chirality operator $\gamma$. Furthermore, in suitable bidegree this operator will be hypoelliptic under 
the finite type condition alluded to in Section~\ref{sec.CRARS}. 

First, recall that the tangential $\overline{\partial}$-operator  of 
Kohn-Rossi~(\cite{KR:EHFBCM},~\cite{Ko:BCM}) can be defined as follows. For any $\eta \in C^{\infty}(M, \Lambda^{p,q})$ its differential uniquely 
decomposes as  
\begin{equation}
    d\eta =\dbarb \eta + \partial_{H}\eta + \theta \wedge \cL_{X_{0}}\eta,
     \label{eq:CR.dbarb}
\end{equation}
where $\dbarb \eta $  (resp.~$\partial_{H}\eta$) is a section of $\Lambda^{p,q+1}$ (resp.~$\Lambda^{p+1,q}$). Moreover, when
$T_{1,0}$ is integrable $\overline{\partial}_{H}^{2}$ vanishes on $(0,q)$-forms, and so that we then get a cochain 
 complex $\overline{\partial}_{H}:C^{\infty}(M,\Lambda^{0,*})\rightarrow C^{\infty}(M,\Lambda^{0,*+1})$. 
 

The operator $Q_{L}$ is defined by
\begin{equation}
         Q_{L}=(\bar\partial_{H}^{*}\bar\partial_{H} - 
     \bar\partial_{H}\bar\partial_{H}^{*}) - \gamma (\bar\partial_{H}^{*}\bar\partial_{H} - 
     \bar\partial_{H}\bar\partial_{H}^{*})\gamma.
          \label{eq:QL.definition}
 \end{equation}

In order to determine the local expression of $Q_{L}$, let $\{Z_{j},Z_{\bar{j}}\}$ be an admissible orthonormal frame of $H\otimes \C$,  and 
let $\{\theta^{j},\theta^{\bar{j}}\}$ be the associated dual coframe of $H^{*}\otimes \C$. In addition,
we let $\varepsilon(\theta^{\bar{j}})$ denote the exterior multiplication by $\theta^{\bar{j}}$ and 
 let $\iota(\theta^{j})$ denote the interior product by $\theta^{j}$ (i.e., $\iota(\theta^{j})$ is the contraction of forms by the 
 vector field $Z_{\bar{j}}$).

 \begin{lemma}\label{lem:CR.gamma-thetaj-iotaj}
For $j,k=1,\ldots,n$ we have 
\begin{gather}
\gamma  \varepsilon(\theta^{\bar{j}})\gamma=  i\iota(\theta^{j}), \qquad
\gamma \iota(\theta^{j})\gamma= -i\varepsilon(\theta^{\bar{j}}),    
    \label{eq:CR.Clifford-gamma1}  \\
\gamma  \varepsilon(\theta^{\bar{j}})\iota(\theta^{k})\gamma=  \iota(\theta^{j})\varepsilon(\theta^{\bar{k}}), \qquad
\gamma \iota(\theta^{j})\varepsilon(\theta^{\bar{k}}) \gamma=  \varepsilon(\theta^{\bar{j}})\iota(\theta^{k}).
    \label{eq:CR.Clifford-gamma2} 
\end{gather}
\end{lemma}
\begin{proof}
  First, since $\gamma^{2}=1$ the equalities $\gamma  \varepsilon(\theta^{\bar{j}})\gamma=i\iota(\theta^{j})$ and $\gamma \iota(\theta^{j})\gamma=  
 -i \varepsilon(\theta^{\bar{j}})$ are equivalent to each other. Moreover, we can deduce from them the equalities~(\ref{eq:CR.Clifford-gamma2}). Therefore, we only have 
  to prove that  $\gamma  \varepsilon(\theta^{\bar{j}})\gamma=i\iota(\theta^{j})$.

  If $J=\{j_{1},\ldots,j_{p}\}$ is an ordered subset of $\{1,\ldots,n\}$ and $j$ is an element of $J^{c}$ such that $j_{k}<j<j_{k+1}$ we let 
  $\tilde{\varepsilon}(j,J)=(-1)^{k}$, so that $\tilde{\varepsilon}(j,J)$ is the signature of the permutation 
  $(j,j_{1},\ldots,j_{p})\rightarrow (j_{1},\ldots,j_{k},j,j_{k+1},\ldots,j_{p})$. Then
  \begin{equation}
      \varepsilon(\theta^{\bar{j}})\theta^{J,\bar{K}}=\left\{ 
      \begin{array}{cl}
          \tilde{\varepsilon}(j,K)\theta^{J,\overline{K\cup\{j\}}}& \text{if $J\not\in K$},  \\
          0 & \text{otherwise}.
      \end{array}\right.
       \label{eq:CR.varepsilon-thetaj}
  \end{equation}
Similarly, we have 
  \begin{equation}
      \iota(\theta^{\bar{j}})\theta^{J,\bar{K}}=\left\{ 
      \begin{array}{cl}
          \tilde{\varepsilon}(j,K\setminus\{j\})\theta^{J,\overline{K\setminus\{j\}}}& \text{if $J\in K$},  \\
          0 & \text{otherwise}.
      \end{array}\right.
       \label{eq:CR.iota-thetaj}
 \end{equation}

Let $J$ and $K$ be ordered subsets of $\{1,\ldots,n\}$ of respective lengths $p$ and $q$. Using~(\ref{eq:CR.gamma-thetaJK}) we get
\begin{equation}
 \gamma   \varepsilon(\theta^{\bar{j}})\gamma  \theta^{J,\bar{K}}= i^{(p+q)^{2}}(-1)^{\frac{n(n-1)}{2}+q(n-p)}\varepsilon(J,J^{c})\varepsilon(K,K^{c})\gamma 
 \varepsilon(\theta^{\bar{j}})\theta^{J^{c},\bar{K^{c}}}.
     \label{eq:CR.gamma-varepsilon-thetaj-gamma1}
\end{equation}
In particular, we see that $\gamma  \varepsilon(\theta^{\bar{j}})\gamma \theta^{J,\bar{K}}=0$ if $j\not \in K$. 

Assume now that $j$ is in $K$.  Then using~(\ref{eq:CR.gamma-thetaJK}), (\ref{eq:CR.varepsilon-thetaj}) and~(\ref{eq:CR.gamma-varepsilon-thetaj-gamma1}) 
we get 
\begin{multline}
 \gamma  
\varepsilon(\theta^{\bar{j}})\gamma \theta^{J,\bar{K}} =     i^{(p+q)^{2}}(-1)^{\frac{n(n-1)}{2}+q(n-p)}\varepsilon(J,J^{c})\varepsilon(K,K^{c})  
\tilde{\varepsilon}(j,K^{c})  \gamma 
     \theta^{J^{c},\overline{K^{c}\cup\{j\}}}\\ 
    =\lambda_{1} \lambda_{2} \lambda_{3}\theta^{J,\overline{K\setminus\{j\}}},
     \label{eq:CR.gamma-varepsilon-thetaj-gamma2}
\end{multline}
where we have let 
\begin{gather}
    \lambda_{1}=i^{(p+q)^{2}+(2n-p-q+1)^{2}}, \qquad
      \lambda_{2}=(-1)^{q(n-p)+(n-q+1)p}\varepsilon(J,J^{c})\varepsilon(J,J^{c}),\\ 
         \lambda_{3}=  \varepsilon(K,K^{c})   \varepsilon(K^{c}\cup\{j\},K\setminus \{j\})  \tilde{\varepsilon}(j,K^{c}).
\end{gather}

Recall that given any integer $m$ the difference $m^{2}-m=m(m-1)$ is always an even number. Thus, 
\begin{equation}
    \lambda_{1}=i^{(p+q)^{2}+(p+q-1)^{2}}=i^{2[(p+q)^{2}-(p+q)]+1}=i.
     \label{eq:CR.lambda1}
\end{equation}
Moreover, as $\varepsilon(J,J^{c})\varepsilon(J,J^{c})=(-1)^{p(n-p)}$ we have
\begin{equation}
    \lambda_{2}=(-1)^{q(n-p)+(n-q)p+p+p(n-p)}=(-1)^{nq-2pq+2np+p-p^{2}}=(-1)^{nq}.
     \label{eq:CR.lambda2}
\end{equation}

Next, set $K=\{k_{1},\ldots,k_{q}\}$ and $K^{c}\cup\{j\}=\{k_{1}',\ldots,k_{n-q+1}'\}$. Then we have $j=k_{l}=k_{l'}'$ for some indices $l$ and 
$l'$. By definition $\varepsilon(K,K^{c})$ is the signature of the permutation 
$(k_{1},\ldots,k_{q},k_{1}',\ldots,\hat{k'}_{l'},\ldots,k'_{n-q+1})\rightarrow (1,\ldots,n)$. This permutation can also 
be seen as the composition of the following 
permutations, 
\begin{multline}
  (k_{1},\ldots,k_{q},k_{1}',\ldots,\hat{k'}_{l'},\ldots,k'_{n-q+1})\\ \rightarrow 
  (j, k_{1},\ldots, \hat{k}_{l},\ldots, k_{q},k_{1}',\ldots,\hat{k'}_{l'},\ldots,k'_{n-q+1})   \\
    \rightarrow (k_{1},\ldots, \hat{k}_{l},\ldots, k_{q},j, k_{1}',\ldots,\hat{k'}_{l'},\ldots,k'_{n-q+1})\\
    \rightarrow (k_{1},\ldots, \hat{k}_{l},\ldots, k_{q},k_{1}',\ldots, k'_{n-q+1}) \rightarrow (1,\ldots,n). 
\end{multline}
The respective signatures of these permutations are
\begin{gather}
   \tilde{\varepsilon}(j,K\setminus \{j\}), \qquad (-1)^{q-1}, \qquad \tilde{\varepsilon}(j,K^{c}),   \\
   \varepsilon(K\setminus\{j\},K^{c}\cup\{j\})=(-1)^{(q-1)(n-q+1)}\varepsilon(K^{c}\cup\{j\},K\setminus \{j\}).
\end{gather}
As $(-1)^{q-1}(-1)^{(q-1)(n-q+1)}=(-1)^{(q-1)(n-q)}=(-1)^{nq-q^{2}-n+q}=(-1)^{n+nq}$ we deduce that
$\varepsilon(K,K^{c})=(-1)^{(q-1)(n-q)+}\tilde{\varepsilon}(j,K\setminus 
\{j\})\tilde{\varepsilon}(j,K^{c})\varepsilon(K^{c}\cup\{j\},K\setminus \{j\})$. Thus,
\begin{equation}
    \lambda_{3}=(-1)^{n+nq}\varepsilon(j,K\setminus\{j\}).
     \label{eq:CR.lambda3}
\end{equation}

Now, combining~(\ref{eq:CR.gamma-varepsilon-thetaj-gamma2}) with (\ref{eq:CR.lambda1}), (\ref{eq:CR.lambda2}) and (\ref{eq:CR.lambda3}) 
gives
\begin{equation}
    \gamma  \varepsilon(\theta^{\bar{j}})\gamma 
\theta^{J,\bar{K}}=i(-1)^{n}\tilde{\varepsilon}(j,K\setminus\{j\})\theta^{J,\overline{K\setminus\{j\}}},
\end{equation}
so using~(\ref{eq:CR.iota-thetaj})  we get
\begin{equation}
    \gamma  \varepsilon(\theta^{\bar{j}})\gamma 
\theta^{J,\bar{K}}=i (-1)^{n}\iota(\theta^{\bar{j}})\theta^{J,\bar{K}}.
\end{equation}
Since $\iota(\theta^{\bar{j}})\theta^{J,\bar{K}}=0$ when $j\not \in K$, this shows that $\gamma  \varepsilon(\theta^{\bar{j}})\gamma 
=i(-1)^{n}\iota(\theta^{\bar{j}})$, completing the proof. 
\end{proof}

In the sequel we let $\op{O}_{H}(0)$ denote a general zeroth order differential operator and we let $\op{O}_{H}(1)$ denote a first order 
differential operator involving only differentiations along $H\otimes \C$ . For instance, seen as differential operators, $Z_{j}$ and $Z_{\bar{j}}$ both are 
$\op{O}_{H}(1)$, but $X_{0}$ is not. Bearing this in mind the following holds.

\begin{proposition}\label{prop:QL.local-form}
 In the local trivialization of $\Lambda^{*,*}$ defined by the orthonormal coframe $\{\theta^{J,\bar{K}}\}$ we have 
     \begin{equation}
         Q_{L}=\sum_{j,k=1}^{n} \left( \varepsilon(\theta^{\bar{j}})\iota(\theta^{k})-\iota(\theta^{j})\varepsilon(\theta^{\bar{k}})\right)  
     ( Z_{\bar{j}}Z_{k}+Z_{j}Z_{\bar{k}}) +           
         \op{O}_{H}(1).
         \label{eq:QL.local-formula}
     \end{equation}
In particular, in the local trivialization of $\Lambda^{*,0}\oplus \Lambda^{*,n}$ we have 
\begin{equation}
      Q_{L}=\pm \sum_{j=1}^n (Z_{j}Z_{\bar{j}}+Z_{\bar{j}}Z_{j})+  \op{O}_{H}(1),
     \label{eq:CR.Qb-local.*0*n}
\end{equation}
where the sign $\pm $ is $-$ on $\Lambda^{*,0}$ and $+$ on $\Lambda^{*,n}$.  
\end{proposition}
\begin{proof}
    Set $Q_{L}':=\bar\partial_{H}^{*}\bar\partial_{H} - 
     \bar\partial_{H}\bar\partial_{H}^{*}$. One can check (see, e.g, \cite{BG:CHM}) that  in the trivialization of $\Lambda^{*,*}$ defined by 
     $\{\theta^{J,\bar{K}}\}$ we have 
\begin{equation}
         \dbarb =\sum_{j=1}^n \varepsilon(\theta^{\bar{j}})Z_{\bar{j}}\qquad  \text{and} \qquad \dbarb^{*}=-\sum_{j=1}^n  \iota(\theta^{j})Z_{j}+\op{O}_{H}(0).
\end{equation}
 Therefore, 
 \begin{multline}
        Q_{L}'=-\sum_{j,k=1}^{n} \iota(\theta^{j})\varepsilon(\theta^{\bar{k}})Z_{j}Z_{\bar{k}}  + 
        \sum_{j,k=1}^{n} \varepsilon( \theta^{\bar{j}})\iota(\theta^{k})Z_{\bar{j}}Z_{k}+\op{O}_{H}(1)\\
     =\sum_{j,k=1}^{n} \left( \varepsilon(\theta^{\bar{j}})\iota(\theta^{k})Z_{\bar{j}}Z_{k}- \iota(\theta^{j})\varepsilon(\theta^{\bar{k}})Z_{j}Z_{\bar{k}}\right) +\op{O}_{H}(1).
      \label{eq:QL.QL'}
 \end{multline}

 Using Lemma~\ref{lem:CR.gamma-thetaj-iotaj} we get
\begin{multline}
   \gamma Q_{L}'\gamma=   \sum_{j,k=1}^{n} \left( \gamma\varepsilon(\theta^{\bar{j}})\iota(\theta^{k})\gamma Z_{\bar{j}}Z_{k}- \gamma 
   \iota(\theta^{j})\varepsilon(\theta^{\bar{k}})\gamma Z_{j}Z_{\bar{k}}\right) +\op{O}_{H}(1) \\
  =  \sum_{j,k=1}^{n} \left( \iota(\theta^{j})\varepsilon(\theta^{\bar{k}}) Z_{\bar{j}}Z_{k}- \varepsilon(\theta^{\bar{j}})\iota(\theta^{k}) Z_{j}Z_{\bar{k}}\right) +\op{O}_{H}(1).
\end{multline}
Combining this with~(\ref{eq:QL.QL'}) then shows that modulo $\op{O}_{H}(1)$-terms we have
 \begin{multline*}
     Q_{L}=Q_{L}'-\gamma Q_{L}'\gamma \\ = 
   \sum_{j,k=1}^{n} \left( \varepsilon(\theta^{\bar{j}})\iota(\theta^{k})Z_{\bar{j}}Z_{k}- 
   \iota(\theta^{j})\varepsilon(\theta^{\bar{k}})Z_{j}Z_{\bar{k}}
     -   \iota(\theta^{j})\varepsilon(\theta^{\bar{k}}) Z_{\bar{j}}Z_{k}+ \varepsilon(\theta^{\bar{j}})\iota(\theta^{k}) Z_{j}Z_{\bar{k}}\right) \\
     =  \sum_{j,k=1}^{n} \left( \varepsilon(\theta^{\bar{j}})\iota(\theta^{k})-\iota(\theta^{j})\varepsilon(\theta^{\bar{k}})\right)  
     ( Z_{\bar{j}}Z_{k}+Z_{j}Z_{\bar{k}}).
 \end{multline*}
 
Now, on $\Lambda^{p,0}$ we have $\varepsilon(\theta^{\bar{j}})\iota(\theta^{k})=0$ and 
$\iota(\theta^{j})\varepsilon(\theta^{\bar{k}})=\delta^{\bar{j}k}$. Therefore, on $(p,0)$-forms we have 
\begin{equation}
    Q_{L}=\sum_{j,k=1}^{n} (- \delta^{\bar{j}k})\left( Z_{\bar{j}}Z_{k}+Z_{j}Z_{\bar{k}}\right) +\op{O}_{H}(1)
    =-\sum_{j=1}^n( Z_{\bar{j}}Z_{j}+Z_{j}Z_{\bar{j}})+\op{O}_{H}(1).
\end{equation}
 Similarly, as on $\Lambda^{p,n}$ we have $\varepsilon(\theta^{\bar{j}})\iota(\theta^{k})=\delta^{\bar{j}k}$ and 
$\iota(\theta^{j})\varepsilon(\theta^{\bar{k}})=0$, we see that on $(p,n)$-forms $Q_{L}=\sum_{j=1}^n ( 
Z_{\bar{j}}Z_{j}+Z_{j}Z_{\bar{j}})+\op{O}_{H}(1)$. The proof is complete.
\end{proof}

\section{Hypoelliptic Properties of $Q_{L}$}\label{sec:hypoelliticity}
From now on we assume that $M$ is compact. This assumption is not essential, but it will simplify the exposition of what follows. In fact,
all the following results can be localized and, as such, they continue to hold  in the non-compact case. 

As $\{Z_{j}\}$ is an admissible frame, we can write $Z_{j}=\frac{1}{\sqrt{2}}(X_{j}-iX_{n+j})$, where $\{X_{j}\}$ is a local 
orthonormal frame of $L$ and $X_{n+j}:=\cJ X_{j}$. Then  using~(\ref{eq:QL.local-formula}) we can check that on $(p,0)$-forms and $(p,n)$-forms
we have
\begin{equation}
    Q_{L}=\pm (X_{1}^{2}+\ldots+X_{2n}^{2}) + \op{O}_{H}(1). 
\end{equation}
This means that, up to sign factor and to an $\op{O}_{H}(1)$-term, on these forms $Q_{L}$ is a sum of squares.  A well-known result of
H\"ormander~\cite{Ho:HSODE} then insures us that when $(M,H)$ has finite type such an operator is hypoelliptic with gain of $2/r$ derivatives, 
where $r$ denotes the minimal number of Lie brackets of vector fields with values of $H$ that are needed to span $TM$. In other words, for all $s \in 
\R$, we have 
\begin{equation}
    Q_{L}u\in L^{2}_{s} \Longrightarrow u \in L^{2}_{s+\frac{2}{r}}.
\end{equation}

Given vector fields $X_{1},\ldots,X_{m}$ spanning $H$ at every point, Folland-Stein~\cite{FS:EDdbarbCAHG} and Rothschild-Stein~\cite{RS:HDONG} introduced
suitable functional spaces to study sums of squares. Namely, for $k=0,1,2,\ldots$ they defined
\begin{equation}
     S_{k}^{2}(M):=\bigcap_{1\leq l\leq k}\{u\in L^{2}; X_{i_{1}}\ldots X_{i_{l}}u\in L^{2}\ \forall i_{j}\in\{1,\ldots,m\}\}, 
     \label{eq:hypoellipticity.Sk2a}
\end{equation}
which is a Hilbert space when endowed with the Hilbert norm, 
\begin{equation}
    \|u\|_{S_{k}^{2}}:=\left( \sum_{1\leq l\leq k}\sum_{i_{1},\ldots,i_{l}}\|X_{i_{1}}\ldots 
    X_{i_{l}}u\|^{2}\right)^{\frac{1}{2}}, \qquad u\in S^{2}_{k}(M). 
     \label{eq:hypoellipticity.Sk2b}
\end{equation}
These definitions also makes sense for sections of any vector bundle over $M$.

If $P$ is a differential operator of order $m$ on $M$, we say that $P$ is \emph{maximal hypoelliptic} if, for all $k \in \N$, 
\begin{equation}
    Pu \in S_{k}^{2} \Longrightarrow u \in S_{k+m}^{2}.
\end{equation}

Rothschild-Stein~\cite{RS:HDONG} proved that if $(M,H)$ has finite type, then a sum of squares is maximal hypoelliptic and we have a continuous inclusion
$S^{2}_{k}\subset L^{2}_{\frac{k}{r}}$. Incidentally, maximal hypoellipticity implies hypoellipticity with gain of $\frac{2}{r}$-derivatives. 

In fact, Rothschild-Stein~\cite{RS:HDONG} and Rothschild-Tartakoff~\cite{RT:PCEBOHT} even constructed
parametrices for sum of squares in a suitable class of singular-integral operators.  These operators enjoy various regularity properties, including mapping continuously 
$S_{k}^{2}$ to $S_{k+2}^{2}$. We refer to~\cite{RS:HDONG} for a thorough account on these properties.  

Summarizing all this we obtain

\begin{proposition}\label{prop:QL-hypoellipticity}
Assume that $(M,H)$ is of finite type. Then on $(p,0)$-forms and $(p,n)$-forms $Q_{L}$ is  maximal hypoelliptic and 
admits a parametrix in the class of singular-integral operators of Rothschild-Stein. 
\end{proposition}

Suppose now that $(M,H)$ is of hypersurface type, i.e., $\codim H=1$. The Levi form of $(M,H)$ is then 
defined as the Hermitian form,
\begin{equation}
    \cL:T_{1,0}\times T_{1,0}\longrightarrow T_{\C}M/(H\otimes \C)
\end{equation}
such that, for all sections $Z$ and $W$ of $T_{1,0}$ and for all $x\in M$, we have
\begin{equation}
    \cL_{x}(Z(x),{W}(x))=[Z,\overline{W}](x) \quad \bmod H_{x}\otimes \C.
\end{equation}
It is not difficult to check that $(M,H)$ is of finite type if and only if $\cL$ does not vanish anywhere. 

On the other hand, when $\codim H=1$ Beals-Greiner~\cite{BG:CHM} (see also~\cite{Ta:NCMA}) constructed a pseudodifferential calculus, the so-called 
Heisenberg calculus, containing a full symbolic calculus allowing us to \emph{explicitly} construct parametrices for sums of squares, as well as for the 
Kohn Laplacian under the so-called condition $Y(q)$ (see~\cite{BG:CHM}). Therefore, we obtain

\begin{proposition}\label{prop:QL.hypoellipticity-codim1}
 1)  If $\codim H=1$ and $\cL$ is non-vanishing, then on $(p,0)$-forms and $(p,n)$-forms  $Q_{L}$ admits a parametrix in the
  Heisenberg calculus.\smallskip
  
  2) If $\dim M=3$ and  $\cL$ is non-vanishing, then in every bidegree $Q_{L}$ is hypoelliptic and admits a parametrix in the
  Heisenberg calculus.
 \end{proposition}

The hypoellipticity properties of $Q_{L}$ show a new phenomenon with respect to what happens for the Kohn Laplacian, i.e., the Laplacian of the 
$\dbarb$-complex,  
\begin{equation}
    \Box_{H}:=\dbarb^{*}\dbarb + \dbarb \dbarb^{*}.
\end{equation}
For CR manifolds of hypersurface type the invertibility in the Heisenberg calculus' sense of the principal symbol of $\Box_{H}$ on $(p,q)$-forms  
is equivalent to the $Y(q)$-condition of Kohn~\cite{Ko:BCM} (see~\cite{BG:CHM}). 

When the CR manifold $(M,H)$ is strictly pseudoconvex 
the condition $Y(q)$ means that we must have $0<q<n$.  In particular, this excludes all the $(p,q)$-forms in dimension 3. Thus, in the strictly pseudoconvex 
case, the operator $Q_{L}$ has an invertible principal symbol precisely on forms where the Kohn Laplacian has not an invertible principal symbol. 

When $(M,H)$ is not strictly pseudoconvex, but is weakly pseudoconvex, then the $Y(q)$-condition always fails. However, if $(M,H)$ has finite type and the Levi form has comparable 
eigenvalues then the Kohn Laplacian is hypoelliptic (see~\cite{Ko:MSHETCROBL} and the references therein). There also are examples of CR manifolds 
whose Levi form does not have comparable eigenvalues and for which the Kohn Laplacian still enjoys nice regularity properties (see, e.g.,~\cite{FKM:HECRMDLF}, 
\cite{NS:dCDBCn}). 

In contrast, the hypoellipticity of $Q_{L}$ on $(p,0)$-forms and $(p,n)$-forms is independent of any convexity property of the Levi form, 
since the sole non-vanishing of $\cL$ is enough to have maximal hypoellipticity. 

\section{Failure of Hypoellipticity on $\bH^{5}$}\label{sec:non-hypoellipticity}
In the previous section we saw that when restricted to $(p,0)$-forms and $(p,n)$-forms $Q_{L}$ is maximal hypoelliptic. In this section we would 
like to explain that that when restricted to other forms the operator $Q_{L}$ may fail to be hypoelliptic. To this end we shall look at the example of $Q_{L}$ on 
the 5-dimensional Heisenberg group $\bH^{5}$ acting on $(0,1)$-forms. 

Notice that the (localized versions) of the notions of hypoellipticity alluded to in the previous section all imply the following usual notion of hypoellipticity
\begin{equation}
    Q_{L}u\in C^{\infty} \Longrightarrow u \in C^{\infty}.
\end{equation}
For homogeneous left-invariant differential operators on $\bH^{5}$ (and more generally on nilpotent graded Lie groups) this can be shown to be equivalent to maximal hypoellipticity 
(see~\cite{Fo:SEFNLG}). In this section we 
shall exhibit a $(0,1)$-form on $\bH^{5}$ which is singular at the origin and annihilates $Q_{L}$. This will prove that $Q_{L}$ is not hypoelliptic on $(0,1)$-forms. 

Throughout this section we will keep on using the notation introduced in Example B of Section~\ref{sec.CRARS} to describe the Heisenberg group. 
Thus $\bH^{5}$ is $\R\times \R^{4}$ equipped 
with the group 
law~(\ref{eq:real-structures.Heisenberg-law}). 
We let $X_{0},\ldots,X_{4}$ be the left-invariant vector fields defined by~(\ref{eq:Examples.Heisenberg-left-invariant-basis}). In this context $H$ is the vector bundle spanned by 
$X_{1},\ldots,X_{4}$, it complex structure $\cJ$ is such that $\cJ X_{j}=X_{2+j}$ and $\cJ X_{2+j}=-X_{j}$ for $j=1,2$, and $L$  
is the vector bundle spanned by $X_{1}$ and $X_{2}$. In addition, we equip $\bH^{5}$ with its Levi metric $g:=\theta^{2}+d\theta(.,\cJ.)$, 
where $\theta=dx^{0}+\sum_{j=1}^{2}(x^{j}dx^{2+j}-x^{2+j}dx^{j})$ is the 
standard contact form of $\bH^{5}$. With respect to this metric $X_{0},\ldots,X_{4}$ form an orthonormal frame of $T\bH^{5}$. 

In the sequel it will be convenient to introduce for $j=1,2$ the complex coordinates $z^{j}=x^{j}+ix^{2+j}$ and $z^{\bar{j}}=x^{j}-ix^{2+j}$, as well as the 
vector fields,
\begin{gather}
   Z_{j}=\frac{1}{\sqrt{2}}(X_{j}-iX_{2+j})=\frac{1}{\sqrt{2}}(\frac{\partial}{\partial z^{j}}+iz^{\bar{j}} \frac{\partial}{\partial x^{0}}),\\ 
   \quad Z_{\bar{j}}=\overline{Z_{j}}=\frac{1}{\sqrt{2}}(\frac{\partial}{\partial z^{\bar{j}}}-iz^{{j}} \frac{\partial}{\partial x^{0}}).
\end{gather}
Then $\{Z_{1},Z_{2},Z_{\bar{1}},Z_{\bar{2}}\}$ is a left-invariant orthonormal frame of $H\otimes \C$. Because of the way the CR and real structures are defined in terms of the 
vector fields $X_{1},\ldots,X_{4}$, this orthonormal frame is admissible in the sense used in the previous section.

Let $\{\theta^{1},\theta^{2},\theta^{\bar{1}}, \theta^{\bar{2}}\}$ be the dual coframe of $\{Z_{1},Z_{2},Z_{\bar{1}},Z_{\bar{2}}\}$. In fact, we can 
check that  $\theta^{j}=\sqrt{2}dz^{j}$ and 
$\theta^{\bar{j}}=\sqrt{2}dz^{\bar{j}}$ for $j=1,2$. Since $\{Z_{1},Z_{2},Z_{\bar{1}},Z_{\bar{2}}\}$ is an admissible orthonormal frame Eq.~(\ref{eq:QL.local-formula}) holds. 
It actually 
holds without a $\op{O}_{H}(1)$ remainder term. Indeed, as $\dbarb=\sqrt{2}\sum \varepsilon(dz^{\bar{j}})Z_{\bar{j}}$ we can check that $\dbarb^{*}=-
\frac1{\sqrt{2}}\sum \iota(dz^{{j}})Z_{{j}}$, where $\iota(dz^{j})$ denotes the contraction by $\frac{d}{dz^{\bar{j}}}$ (it agrees  with that by 
$\sqrt{2}Z_{\bar{j}}$ on $H^{*}\otimes \C$). Following the lines of the proof of Proposition~\ref{prop:QL.local-form} we then see that no remainder terms 
are involved anymore. Thus, 
\begin{equation}
    Q_{L}=\sum_{j,k=1,2} \left( \varepsilon(dz^{\bar{j}})\iota(dz^{k})-\iota(dz^{j})\varepsilon(dz^{\bar{k}})\right)  
     ( Z_{\bar{j}}Z_{k}+Z_{j}Z_{\bar{k}}).
\end{equation}

Next, observe that 
\begin{equation}
    \varepsilon(dz^{\bar{j}})\iota(dz^{k})dz^{\bar{l}}=\delta^{k\bar{l}}dz^{\bar{j}} \quad \text{and} \quad 
    \iota(dz^{j})\varepsilon(dz^{\bar{k}})=(1-\delta^{k\bar{l}})(\delta^{j\bar{k}}dz^{\bar{l}}-\delta^{j\bar{l}}dz^{\bar{k}}).
\end{equation}
Using this we can check that, with respect to the frame $\{dz^{\bar{1}},dz^{\bar{2}}\}$, on $(0,1)$-forms $Q_{L}$ takes the form,
\begin{equation}
    Q_{L}=\left( 
    \begin{array}{cc}
        \Delta_{1}-\Delta_{2} &  T \\
        T&  \Delta_{2}-\Delta_{1}
    \end{array}\right),
\end{equation}
where we have set $\Delta_{j}:=Z_{\bar{j}}Z_{j}+Z_{j}Z_{\bar{j}}$, $j=1,2$, and $T:=2(Z_{\bar{1}}Z_{2}+Z_{1}Z_{\bar{2}})$, and we also have used the fact 
that $[Z_{1},Z_{\bar{2}}]=[Z_{2},Z_{\bar{1}}]=0$. 

Let $\cF_{0}u:=\int_{-\infty}^{\infty} e^{-ix^{0}.\xi_{0}}udx^{0}$ denote the Fourier transform with respect to the variable $x_{0}$ on $\cS'$. We shall now look at 
$Q_{L}$ under $\cF_{0}$. Notice this is merely the same as looking at $Q_{L}$ under the irreducible representations of $\bH^{5}$. To this end  
we shall use the symbol $\hat{~}$ to denote the conjugation by $\cF_{0}$. We have 
\begin{gather}
     \hat{Z}_{j}=\frac{1}{\sqrt{2}}\cF_{0}(\frac{\partial}{\partial z^{j}}+iz^{\bar{j}} \frac{\partial}{\partial x^{0}})\cF_{0}^{-1}=
     \frac{1}{\sqrt{2}}(\frac{\partial}{\partial z^{j}}-z^{\bar{j}}\xi_{0}),\\
          \hat{Z}_{\bar{j}}=\frac{1}{\sqrt{2}}\cF_{0}(\frac{\partial}{\partial z^{\bar{j}}}-iz^{{j}} \frac{\partial}{\partial x^{0}})\cF_{0}^{-1}=
     \frac{1}{\sqrt{2}}(\frac{\partial}{\partial z^{\bar{j}}}+z^{{j}}\xi_{0}).
\end{gather}
Using this we can check that $2\hat{Z}_{\bar{j}}\hat{Z}_{{j}}$ is equal to
\begin{multline}
  (\frac{\partial}{\partial z^{\bar{j}}}+z^{{j}}\xi_{0})(\frac{\partial}{\partial z^{j}}-z^{\bar{j}}\xi_{0})
  =\frac{\partial^{2}}{\partial z^{\bar{j}} \partial z^{{j}}} +\xi_{0}\left(-\frac{\partial}{\partial 
   z^{\bar{j}}}z^{\bar{j}}+z^{j}\frac{\partial}{\partial z^{{j}}}\right) -z^{j}z^{\bar{j}}\xi_{0}^{2} \\
    =H_{j}+\xi_{0}R_{j}-\xi_{0},
\end{multline}
where we have set $H_{j}:=\frac{\partial^{2}}{\partial z^{\bar{j}} \partial z^{{j}}}-|z^{j}|^{2}\xi_{0}^{2}$ and 
$R_{j}:=z^{j}\frac{\partial}{\partial z^{{j}}}-z^{\bar{j}}\frac{\partial}{\partial z^{\bar{j}}}$. Similarly,
\begin{equation}
  2\hat{Z}_{{j}}\hat{Z}_{\bar{j}}=  H_{j}+\xi_{0}R_{j}+\xi_{0}.
\end{equation}
We also have 
\begin{equation}
    \hat{T}=(\frac{\partial}{\partial z^{\bar{1}}}+z^{{1}}\xi_{0}) (\frac{\partial}{\partial z^{2}}-z^{\bar{2}}\xi_{0}) + 
    (\frac{\partial}{\partial z^{1}}-z^{\bar{1}}\xi_{0})(\frac{\partial}{\partial z^{\bar{2}}}+z^{{2}}\xi_{0}).
     \label{eq:QLH5.hatT}
\end{equation}
Therefore, on $(0,1)$-forms we have 
\begin{equation}
    Q_{L}=\left( 
    \begin{array}{cc}
       H_{1}-H_{2} +\xi_{0}(R_{1}-R_{2})&  \hat{T} \\
        \hat{T}&  H_{2}-H_{1} +\xi_{0}(R_{2}-R_{1})
    \end{array}\right),
     \label{eq:QLH5.hatQL}
\end{equation}
with $\hat{T}$ given by~(\ref{eq:QLH5.hatT}). 

In the sequel we set $z=(z^{1},z^{2})$ and we consider the $(0,1)$-form,
\begin{equation}
\hat{\omega}:=\hat{u} dz^{\bar{1}}, \qquad   \hat{u}(\xi_{0},z):=\exp\left(-|\xi_{0}| |z|^{2}\right).
\end{equation}
Notice that $\hat{u}$ is a ground state for the harmonic oscillators $H_{j}$ and annihilates the rotation generators $R_{j}$, namely, 
\begin{equation}
    H_{1}\hat{u}=H_{2}\hat{u}=|\xi_{0}|\hat{u} \quad \text{and} \quad  R_{1}\hat{u}=R_{2}\hat{u}=0.
\end{equation}
In addition $\hat{u}$ also annihilates $(\frac{\partial}{\partial z^{j}}-z^{\bar{j}}\xi_{0})\hat{u}$ for $\xi_{0}\geq 0$ and $(\frac{\partial}{\partial 
z^{\bar{j}}}+z^{{j}}\xi_{0})\hat{u}=0$ for $\xi_{0}\leq 0$, and so using~\ref{eq:QLH5.hatT} we see that
\begin{equation}
    \hat{T}\hat{u}=0.
\end{equation}
Combining all this with~(\ref{eq:QLH5.hatQL}) we get
\begin{equation}
    \hat{Q}_{L}\hat{\omega}=\left(H_{1}-H_{2} +\xi_{0}(R_{1}-R_{2})\right)\hat{u}dz^{\bar{1}}+\hat{T}\hat{u}dz^{\bar{2}}=0.
     \label{eq:QLH5.hatQLhatu}
\end{equation}

Next, the inverse transform $ u:=\cF^{-1}_{0}\hat{u}$ is equal to
\begin{multline}
   \frac{1}{2\pi}\int_{-\infty}^{\infty}e^{ix^{0}.\xi_{0}}e^{-|z|^{2}|\xi_{0}|}d\xi_{0}= \frac{1}{2\pi}\left(\int_{0}^{\infty}e^{ix^{0}.\xi_{0}}e^{-|z|^{2}\xi_{0}}d\xi_{0}+  
    \int_{0}^{\infty}e^{-ix^{0}.\xi_{0}}e^{-|z|^{2}\xi_{0}}d\xi_{0} \right)\nonumber \\
    =  \frac{1}{2\pi}\left( \frac{1}{ix^{0}-|z|^{2}} +  \frac{1}{-i x^{0}-|z|^{2}}\right)
    = \frac{-1}{\pi}\frac{|z|^{2}}{|x^{0}|^{2}+|z|^{4}}. 
\end{multline}
Notice that $u$ is homogeneous of degree $-2$ with respect to the dilations~(\ref{eq:RealStructure.Heisenberg-dilations}). In particular $u$ is singular at the origin. 

Set $\omega=\cF_{0}^{-1}\hat{\omega}= udz^{\bar{1}}$. In view of~(\ref{eq:QLH5.hatQLhatu}) we have 
\begin{equation}
    Q_{L}\omega=\cF_{0}^{-1}\hat{Q}_{L}\hat{\omega}=0.
\end{equation}
Therefore, we see that, although $\omega$ is  not smooth at the origin, $Q_{L}\omega$ is smooth everywhere. This shows that $Q_{L}$ is not hypoelliptic on 
$(0,1)$-forms. 

In fact, the same arguments as above also show that the forms $udz^{1}\wedge dz^{\bar{1}}$ and $udz^{1}\wedge dz^{2}\wedge dz^{\bar{1}}$ too 
annihilate $Q_{L}$. Therefore $Q_{L}$ is not hypoelliptic on $(1,1)$-forms or $(2,1)$-forms. This shows that it is \emph{only} on $(p,0)$-forms and $(p,2)$-forms 
that $Q_{L}$ is hypoelliptic.

\end{document}